\documentclass[runningheads,a4paper]{llncs}
\usepackage{amssymb}
\usepackage{algorithmic}
\usepackage{caption}
\usepackage{graphicx}
\usepackage{epsfig}
\usepackage{amsmath}
\usepackage{hyperref}
\usepackage{mathtools}

\usepackage{amsfonts}
\usepackage{amssymb}
\usepackage{bm}
\usepackage{listings}
\usepackage[usenames]{color}
\usepackage[linesnumbered, ruled, lined, boxed]{algorithm2e}
\renewcommand{\algorithmicrequire}{\textbf{Input: }}
\renewcommand{\algorithmicensure}{\textbf{Output: }}
\usepackage{tikz}
\usetikzlibrary{automata, positioning, arrows}
\SetAlgorithmName{Program}{heuristic}{List of Heuristics}

\newcommand{\oomit}[1]{}	
\newtheorem{assumption}{Assumption}

\begin{document}


\title{Robust Non-termination Analysis of Numerical Software \thanks{This work from Bai Xue is funded by
CAS Pioneer Hundred Talents Program under grant No. Y8YC235015, and from Naijun Zhan is funded partly by NSFC under grant No.
61625206 and 61732001, by ``973 Program" under grant No. 2014CB340701,
and by the CAS/SAFEA International Partnership Program for Creative Research Teams, and and from  Yangjia Li is funded by NSFC under grant No.
61502467.}}
\author{
Bai Xue \inst{1`} and Naijun Zhan \inst{1,2} and Yangjia Li \inst{3,1} and Qiuye Wang \inst{1,2}}
\institute{$~^1$ State Key Lab. of Computer Science, Institute of Software, CAS,  China \\
    $~^2$  University of Chinese Academy of Sciences, CAS,  China \\
    $~^3$  University of Tartu, Estonia}

%


%
%

\toctitle{Lecture Notes in Computer Science}
\tocauthor{Authors' Instructions}
\maketitle

\begin{abstract}
Numerical software is widely used in safety-critical systems such as aircrafts, satellites, car engines and many other fields, facilitating dynamics control of such systems in real time. It is therefore absolutely necessary to verify their correctness.
Most of these verifications are conducted under ideal mathematical models, but their real executions may not follow the models exactly. Factors that are abstracted away in models such as rounding errors can change behaviors of systems essentially. As a result, guarantees of verified properties despite the present of disturbances are needed. In this paper, we attempt to address this issue of nontermination analysis of numerical software. Nontermination is often an unexpected behaviour of computer programs and may be problematic for applications such as real-time systems with hard deadlines. We propose a method for robust conditional nontermination analysis that can be used to under-approximate the maximal robust nontermination input set for a given program. Here robust nontermination input set is a set from which the program never terminates regardless of the aforementioned disturbances. Finally, several examples are given to illustrate our approach.

\end{abstract}

\begin{keywords}
Numerical software; nontermination analysis; robust verification
\end{keywords}
\section{Introduction}
\label{Int}
Software is ubiquitous in mission-critical and safety-critical industrial infrastructures as it is, in principle, the most effective way to manipulate complex systems in real time. However, many computer scientists and engineers have experienced costly bugs in embedded software. Examples include the failure of the Ariane 5.01 maiden flight (due to an overflow caused by an unprotected data conversion from a too large 64-bit floating point to a 16-bit signed integer value), the failure of the Patriot missile during the Gulf war (due to an accumulated rounding error), the loss of Mars orbiter(due to a unit error). Those examples indicate that mission-critical and safety-critical software may be far from being safe \cite{cousot2010}. It is therefore absolutely necessary to prove the correctness of software by using formal, mathematical techniques that enable the development of correct and reliable software systems.



The dominant approach to the verification of programs is called \emph{Floyd-Hoare-Naur inductive assertion approach} \cite{Floyd67,Hoare69,Naur66}. It uses \emph{pre-} and \emph{post-condition} to specify the condition of initial states and the property that should be satisfied by terminated states, and use \emph{Hoare logic} to reason about properties of programs.
The hardest parts of this approach are \emph{invariant generation} and \emph{termination analysis}. It is well-know that the termination or nontermination problem is undecidable, and even not semi-decidable in general. Thus, more practical approaches include present some sufficient conditions for termination, or some sufficient conditions for nontermination, or put these two types of conditions in parallel, or prove the decidability for some specific families of programs, e.g.,  \cite{gupta2008,velroyen2008,giesl2014,Chen2014,rebiha2014,Larraz2014,li2017,borralleras2017}.

On the other hand, most of these verifications are conducted under ideal mathematical models, but their real executions may not follow the models exactly. Factors that are abstracted away in models such as rounding errors can change behaviors of systems essentially. As a result, guarantees of verified properties despite the present of disturbances are needed. We notice that this problem affects most existing termination/nontermination as well.

In \cite{FAC11}, the authors presented the following example:
   \begin{example}
   Consider a simple loop
   \[ \text{Q1:} \quad \mathbf{while} ~ (B \mathbf{x} >\mathbf{0}) ~ \{ \mathbf{x} \, := \, A \mathbf{x} \},\]
 where $A = \left(\begin{matrix} 2& -3 \\ -1 & 2 \end{matrix}\right)$,
     $B = \left(\begin{matrix} 1 & b \\ -1 & b \end{matrix}\right)$ with
     $b= -\frac{1127637245}{651041667} = -\sqrt{3} + \epsilon \sim -
     1.732050807$.

     So we have $\epsilon = \sqrt{3} - (-\frac{1127637245}{651041667}) >0$ a small positive number.
     Here we take 10 decimal digits of precision.
 \end{example}
According to the termination decidability result on simple loops proved in \cite{Tiwari04}, $\text{Q1}$ terminates based on exact computation. But unfortunately, it does not terminate in practice as the core decision procedure given in \cite{Tiwari04} invokes a procedure to compute \emph{Jordan normal form} based on numeric computation for a given matrix, and thus the floating error has to be taken into account.
In order to address this issue, a symbolic decision procedure was proposed in \cite{FAC11}. However, a more interesting and challenging issue is to find a systematic way to take all possible disturbances into account during conducting termination and non-termination analysis in practical numerical implementations.

In this paper we attempt to address this challenge, and propose a framework for robust nontermination analysis for numerical software based on control theory as in \cite{roozbehani2013}, which proposes a control-theoretic framework based on Lyapunov invariants to conduct verification of numerical software.
Non-termination analysis proves that programs, or parts of a program, do not terminate. Non-termination is often an unexpected behaviour of computer programs and implies the presence of a bug. If a nonterminating computation occurs, it may be problematic for applications such as real-time systems with hard deadlines or situations when minimizing workload is important.  \oomit{The robust conditional nontermination is the problem of deciding the robust nontermination set from which the given program starting can not terminate, regardless of the actions of uncertain inputs accounting for round-off errors and additive perturbations from real-world phenomena to hardware.}
In this paper, computer programs of interest are restricted to a class of computer programs composed of a single loop with a complicated switch-case type loop body. These programs can also be viewed as a constrained piecewise discrete-time dynamical system with time-varying uncertainties.
We reformulate the problem of determining robust conditional nontermination as finding the maximal robust nontermination input set of the corresponding dynamical system, and characterize that set using a value function, which turns out to be a solution to a suitable mathematical equation. In addition, when the dynamics of the piecewise discrete-time system in each mode is polynomial and the state and uncertain input constraints are semi-algebraic, the optimal control problem is relaxed as a semi-definite programming problem, to which its polynomial solution forms an inner-approximation of the maximal robust nontermination input set when exists. Such relaxation is sound but incomplete. Finally, several examples are given to illustrate our approach.

It should be noticed that the concept of robust nontermination input sets is essentially equivalent to the maximal robustly positive invariants in control theory. Interested readers can refer to, e.g., \cite{blanchini2008set,tahir2012,roozbehani2013,schaich2015}. Computing the maximal robustly positive invariant of a given dynamical system is still a long-standing and challenging problem not only in the community of control theory. Most existing works on this subject focus on linear systems, e.g. \cite{rakovic2005,kouramas2005,tahir2012,schaich2015,trodden2016}. Although some methods have been proposed to synthesize positively invariants for nonlinear systems, e.g., the barrier certificate generation method as in \cite{PrajnaJ04,prajna2007} and the region of attraction generation method as in \cite{jarvis2003,topcu2008,luk2015,giesl2015review}. These methods, however, resort to bilinear sum-of-squares programs, which are notoriously hard to solve. In order to solve the bilinear sum-of-squares programs, a commonly used method is to employ some form of alteration (e.g.,\cite{jarvis2003,topcu2010,luk2015}) with a feasible initial solution to the bilinear sum-of-squares program. Recently,\cite{sassi2012controller,sassi2014} proposed linear programming based methods to synthesize maximal (robustly) positive polyhedral invariants. Contrasting with aforementioned methods, in this paper we propose a semi-definite programming based method to compute semi-algebraic invariant. Our method does not require an initial feasible solution.

\textit{Organization of the paper.} The structure of this paper is as follows. In Section \ref{Pre}, basic notions used throughout this paper and the problem of interest are introduced. Then we elucidate our approach for performing conditional non-termination analysis in Section \ref{IG}. After demonstrating our approach on several illustrating  examples in Section \ref{ex}, we discuss related work in Section \ref{RW} and finally conclude this paper in Section \ref{con}.

\section{Preliminaries}
\label{Pre}
In this section we describe the programs which are considered in this paper and we explain how to analyze them through their representation as piecewise discrete-time dynamical systems.

The following basic notations will be used throughout the rest of this paper: $\mathbb{N}$ stands for the set of nonnegative integers and $\mathbb{R}$ for
the set of real numbers; $\mathbb{R}[\cdot]$ denotes the ring of polynomials in variables given by the argument, $\mathbb{R}_d[\cdot]$ denotes the vector space of real multivariate polynomials of degree $d$, $d\in\mathbb{N}$. Vectors are denoted by boldface letters.

\subsection{Computer Programs of Interest}
In this paper the computer program of interest, as described in \textbf{Program} \ref{alg}, is composed of a single loop with a possibly complicated switch-case type loop body, in which variables $\bm{x}=(x_1,\ldots,x_n)$ are assigned using parallel assignments $(x_1,\ldots,x_n):=\bm{f}(x_1,\ldots,x_n,d_1,\ldots,d_m)$, where $\bm{d}=(d_1,\ldots,d_m)$ is the vector of uncertain inputs, of which values are sampled nondeterministically from a compact set, i.e. $(d_1,\ldots,d_m) \in D$, such as round-off errors in performing computations. 
 The form of programs under consideration is given in \textbf{Program} \ref{alg}.
\begin{algorithm}
$\bm{x}:=\bm{x}_0$;\tcc{$\bm{x}_0\in \Omega$}
\While{$\bm{x}\in X_0$}{
    \tcc{$\bm{d} \in D$}
\If {$\bm{x}\in X_1$}{$\bm{x}:=\bm{f}_1(\bm{x},\bm{d})$;}
\ElseIf{$\bm{x}\in X_2$}{$\bm{x}:=\bm{f}_2(\bm{x},\bm{d})$;}
$\ldots$\\
\ElseIf{$\bm{x} \in X_k$}{$\bm{x}:=\bm{f}_k(\bm{x},\bm{d})$;}
}
\caption{Computer Programs of Interest}
\label{alg}
\end{algorithm}
In \textbf{Program} \ref{alg}, \oomit{$\bm{x}=(x_1,\ldots,x_n)\in \mathbb{R}^n$ is the vector of the program variables. $\bm{d}=(d_1,\ldots,d_m)$, sampled nondeterministically from $D$, represent the program's disturbance inputs such as round-off errors in performing computations, where} $D=\{\bm{d}\mid \bigwedge_{i=1}^{n_{k+1}} h_{k+1,i}(\bm{d})\leq 0\}$ is a compact set in $\mathbb{R}^m$ and $h_{k+1,i}: \mathbb{R}^m \mapsto \mathbb{R}$, is continuous over $\bm{d}$. $\Omega \subseteq \mathbb{R}^n$ stands for
the initial condition on inputs;
 $X_0=\{\bm{x}\in \mathbb{R}^n\mid \bigwedge_{i=1}^{n_{0}} [h_{0,i}(\bm{x})\leq 0]\}$ stands for the loop condition, which
 is a compact set in $\mathbb{R}^n$; $X_j=\{\bm{x}\in \mathbb{R}^n\mid \bigwedge_{i=1}^{n_j}[h_{j,i}(\bm{x})\rhd 0]\}$, $j=1,\ldots,k$,
 stands for the $j-$th branch conditions, where $\rhd\in \{\leq, <\}$. $h_{j,i}: \mathbb{R}^n \mapsto \mathbb{R}$, $j=0,\ldots,k$, $i=1,\ldots,n_j$, $\bm{f}_l: \mathbb{R}^n \times D\mapsto \mathbb{R}^n$, $l=1,\ldots,k$, are continuous functions over $\bm{x}$ and over $(\bm{x},\bm{d})$ respectively. Moreover, $\{X_1,\ldots, X_k\}$ forms a complete partition of
 $\mathbb{R}^n$, i.e. $X_i\cap X_j=\emptyset$ for $\forall i\neq j$, where $i,j\in \{1,\ldots,k\}$, and $\cup_{j=1}^k X_j=\mathbb{R}^n$.

As described in \textbf{Program} \ref{alg},  an update of the variable $\bm{x}$ is executed by the $i$-th branch $\bm{f}_i:\mathbb{R}^n\times D\mapsto \mathbb{R}^n$ if and only if the current value of $\bm{x}$ satisfies the $i$-th branch condition $X_i$.

\subsection{Piecewise Discrete-time Systems}
\label{PPS}
In this subsection we interpret \textbf{Program} \ref{alg} as a constrained piecewise discrete-time dynamical system with uncertain inputs.
Formally,
\begin{definition}
\label{ps}
A constrained piecewise discrete-time dynamical system (PS) is a quintuple $(\bm{x}_0,X_0,\mathcal{X},D,\mathcal{L})$ with
\begin{enumerate}
\item[-] $\bm{x}_0\in \Omega$ is the condition on initial states;
\item[-] $X_0\subseteq \mathbb{R}^n$ is the domain constraint, which is
a compact set. A path can evolve complying with
   the discrete dynamics only if
its current state is in $X_0$;
\item[-] $\mathcal{X}:=\{X_i,i=1,\ldots,k\}$ with $X_i$ as interpreted in \textbf{\bf{Program}} $\ref{alg}$;
\item[-] $D\subseteq \mathbb{R}^m$ is the set of  uncertain inputs;
\item[-] $\mathcal{L}:=\{\bm{f}_i(\bm{x},\bm{d}),i=1,\ldots,k\}$ is the family of the continuous functions $\bm{f}_i(\bm{x},\bm{d}): X_i\times D\mapsto \mathbb{R}^n$,
\end{enumerate}

\end{definition}

In order to enhance the understanding of PS, we use the following figure, i.e. Fig. \ref{fig:another_label}, to illustrate it further.
\begin{figure}[ht]
	\centering
	\begin{tikzpicture}
		\node[state, initial] (p1) {$s_0$};
		\node[left of=p1] (p0) {};
		\node[state, accepting, below of=p1] (p2) {$s_1$};
		\node[state, right of=p1, xshift=3cm] (p3) {$s_2$};
		\node (pdots) at (2.5, 1.3) {\huge \vdots};

		\draw (p0) edge[] node[above]{\scriptsize $\bm{x}:=\bm{x}_0$} (p1);
		\draw (p1) edge[left] node[align=center]{\scriptsize $\bm{x} \notin X_0$, \\ \scriptsize $\bm{x}:=\bm{x}$} (p2)
		(p1) edge[below, bend right] node[]{\scriptsize $\bm{x} \in X_0, \bm{x}:=\bm{x}$} (p3);
		\draw (p3) edge[] node[above]{\scriptsize $\bm{x} \in X_1, \bm{x}:=\bm{f}_1(\bm{x}, \bm{d})$} (p1)
			  (p3) edge[bend right=20] node[above]{\scriptsize $\bm{x} \in X_2, \bm{x}:=\bm{f}_2(\bm{x}, \bm{d})$} (p1)
			  (p3) edge[bend right=60] node[above]{\scriptsize $\bm{x} \in X_k, x:=\bm{f}_k(\bm{x}, \bm{d})$} (p1);

	\end{tikzpicture}
	\caption{An illustrating graph of PS}
	\label{fig:another_label}
\end{figure}
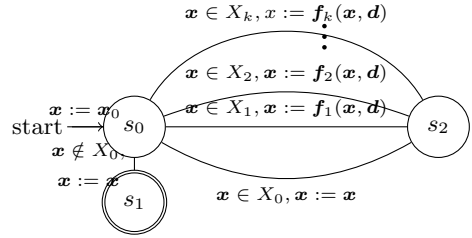
From now on, we associate a PS representation to each program of the form  \textbf{Program} \ref{alg}. Since a program may admit several PS representations, we choose one of them, but the choice does not change the results provided in this paper.

\begin{definition}
An input policy $\pi$ is an ordered sequence $\{\bm{\pi}(i),i\in \mathbb{N}\}$, where $\bm{\pi}(\cdot): \mathbb{N}\mapsto D$,
and $\Pi$ is defined as the set of input policies, i.e. $\Pi=\{\pi\mid \bm{\pi}(\cdot): \mathbb{N}\mapsto D\}$.
\end{definition}
   If an input policy $\pi$ makes \textbf{Program} \ref{alg} non-terminated from an initial state $\bm{x}_0$, then the trajectory
     $\bm{x}_{\bm{x}_0}^{\pi}: \mathbb{N}\mapsto \mathbb{R}^n$
     from $\bm{x}_0$ following the discrete dynamics is defined by
     \begin{equation}
\label{path}
 \bm{x}_{\bm{x}_0}^{\pi}({l+1})=\bm{f}(\bm{x}_{\bm{x}_0}^{\pi}(l),\bm{\pi}(l)),
\end{equation}
where  $\bm{x}_{\bm{x}_0}^{\pi}(0)  =\bm{x}_0$,
  $\forall l \in \mathbb{N}. \bm{x}_{\bm{x}_0}^{\pi}(l) \in X_0$, and \[\bm{f}(\bm{x},\bm{d})=1_{X_1}\cdot\bm{f}_1(\bm{x},\bm{d})+\cdots+1_{X_k}\cdot\bm{f}_k(\bm{x},\bm{d})\]
with $1_{X_i}: X_i\mapsto \{0,1\}$, $i=1,\ldots,k$, representing the indicator function of the set $X_i$,  i.e.
\[1_{X_i}:=\begin{cases}
   1, \quad \text{if }\bm{x}\in X_i,\\
   0, \quad \text{if }\bm{x}\notin X_i.
\end{cases}\]
Consequently, \textbf{Program} \ref{alg} is said to be robust nontermination starting from an initial state $\bm{x}_0\in\Omega$ if for any input policy $\pi\in \Pi$, $\forall l\in \mathbb{N}.~~\bm{x}_{\bm{x}_0}^{\pi}(l)\in X_0$ holds. Formally, 

\begin{definition}
\label{cp}
A program of \textbf{\bf{Program}} $\ref{alg}$ is said to be robust non-terminating w.r.t. an initial state $\bm{x}_0\in X_0$, if
\begin{equation}
\label{path1}
\forall \pi \in \Pi.~\forall l \in \mathbb{N}.~~\bm{x}_{\bm{x}_0}^{\pi}(l)\in X_0.
\end{equation}
\end{definition}

Now, we define our problem of deciding a set of initial states rendering \textbf{Program} \ref{alg}  robust non-termination.
\begin{definition}[Robust Nontermination Set]
\label{RNS}
A set $\Omega$ of initial states in $\mathbb{R}^n$ is a robust nontermination set for a program $P$ of the form \textbf{\bf{Program}} $\ref{alg}$ if
$P$ is robustly non-terminating w.r.t. $\bm{x}_0$ for any $\bm{x}_0\in \Omega$. We call $\{\bm{x}_0 \in \mathbb{R}^n
  \mid P$ is robustly non-terminating w.r.t. $\bm{x}_0\}$
 \emph{the maximal robust non-termination set}, denoted by $\mathcal{R}_0$.
 \oomit{, where
\begin{equation}
\label{R0}
\mathcal{R}_0:=\{\bm{x}_0\mid \forall \pi\Rightarrow \bm{x}_{\bm{x}_0}^{\pi}(l)\in X_0, \forall l\in \mathbb{N}\},
\end{equation}
$\pi$ is the input policy as defined in \eqref{policy} and $\bm{x}_{\bm{x}_0}^{\pi}(\cdot)$ is defined in \eqref{path}. Note that throughout this paper, we assume that $\mathcal{R}_0\neq \emptyset$.}
\end{definition}

From Definition \ref{RNS}, we observe that $\mathcal{R}_0$ is a subset of $X_0$ such that all runs of \textbf{Program} \ref{alg} starting from it can not breach it forever, i.e. if $\bm{x}_0\in \mathcal{R}_0$, $\bm{f}(\bm{x}_0,\bm{d})\in \mathcal{R}_0$ for $\forall \bm{d}\in D$. 
Therefore, the set $\mathcal{R}_0$ is equivalent to the maximal robust positively invariant for PS \eqref{path} in control theory.  For the formal concept of maximal robust positively invariant, please refer to, e.g., \cite{blanchini2008set,tahir2012,schaich2015}. 

\section{Robust Non-Termination Set Generation}
\label{IG}
In this section we elucidate our approach of addressing the problem of robust conditional nontermination for \textbf{Program} \ref{alg}, i.e. synthesizing robust non-termination sets as presented in Definition \ref{RNS}. For this sake, we firstly in Subsection \ref{charac} characterize the maximal robust non-termination set $\mathcal{R}_0$ by means of the value function, which is a solution to  a mathematical equation. Any  solution to this optimal control problem generates a robust non-termination set. Then, in the case that $\bm{f}_i$, $i=1,\ldots,k$, is polynomial over $\bm{x}$ and $\bm{d}$, and the constraint sets over $\bm{x}$ and $\bm{d}$, i.e. $X_j$, $j=0,\dots,k$, and $D$, are of the basic semi-algebraic form, the semi-definite program arising from sum-of-squares decompositions facilitates the gain of inner-approximations $\Omega$ of $\mathcal{R}_0$ via solving the relaxation of the derived optimal control problem in Subsection \ref{AA}.



\subsection{Characterization of $\mathcal{R}_0$}
\label{charac}
In this subsection, we firstly introduce the value function to characterize the maximal robust nontermination set $\mathcal{R}_0$ and then formulate it as a solution to a constrained optimal control problem.

For $\bm{x}_0\in \mathbb{R}^n$, the value function $V:\mathbb{R}^n\mapsto \mathbb{R}$ is defined by:
\begin{equation}
\label{vf}
V(\bm{x}_0):=\sup_{\pi\in \Pi}\sup_{l\in \mathbb{N}}\max_{j\in \{1,\ldots,n_0\}}\big\{h_{0,j}(\bm{x}_{\bm{x}_0}^{\pi}(l))\big\}.
\end{equation}
Note that $V(\bm{x}_0)$ may be neither continuous nor semi-continuous. (A function $V': X'\mapsto \mathbb{R}$ is lower semicontinuous iff for any $y\in \mathbb{R}$, $\{\bm{x}\in X'\mid V'(\bm{x})\geq y\}$ is open, e.g., \cite{Bourbaki2013}.)
%


The following theorem shows the relation between the value function $V$ and the maximal robust nontermination set $\mathcal{R}_0$, that is, the zero sublevel set of $V(\bm{x}_0)$ is equal to the maximal robust nontermination set $\mathcal{R}_0$.  
\begin{theorem}
\label{sets}
$\mathcal{R}_0=\{\bm{x}_0\in \mathbb{R}^n\mid V(\bm{x}_0)\leq 0\}$, where $\mathcal{R}_0$ is the maximal robust nontermination set as in Definition \ref{RNS}.
\end{theorem}
\begin{proof}
Let $\bm{y}_0\in \mathcal{R}_0$.  According to Definition~\ref{RNS}, we have that
\begin{equation}
\label{ineq}
\forall i\in \mathbb{N}.~\forall \pi\in \Pi.~\forall j\in \{1,\ldots,n_0\}.~h_{0,j}(\bm{x}_{\bm{y}_0}^{\pi}(i)) \leq 0
\end{equation}
holds.
Therefore, $V(\bm{y}_0)\leq 0$ and thus $\bm{y}_0\in \{\bm{x}_0\mid V(\bm{x}_0)\leq 0\}$.

On the other side, if $\bm{y}_0\in \{\bm{x}_0\in \mathbb{R}^n\mid V(\bm{x}_0)\leq 0\}$, then  $V(\bm{y}_0)\leq 0$, implying that \eqref{ineq} holds. Therefore, $\bm{y}_0\in \mathcal{R}_0$.

This concludes that $\mathcal{R}_0=\{\bm{x}_0\in \mathbb{R}^n\mid V(\bm{x}_0)\leq 0\}$.
\end{proof}

From Theorem \ref{sets},  the maximal robust nontermination set $\mathcal{R}_0$ could be constructed by computing $V(\bm{x}_0)$, which satisfies the dynamic programming principle as presented in Lemma \ref{dp}.
\begin{lemma}
\label{dp}
 For $\forall \bm{x}_0\in \mathbb{R}^n$ and $\forall l\in\mathbb{N}$, we have:
\begin{equation}
\label{dp1}
\begin{split}
V(\bm{x}_0)=\sup_{\pi\in \Pi}\max\big\{&V(\bm{x}_{\bm{x}_0}^{\pi}(l)),\sup_{i\in [0,l)\cap\mathbb{N}}\max_{j\in \{1,\ldots,n_0\}} h_{0,j}(\bm{x}_{\bm{x}_0}^{\pi}(i))\big\}.
\end{split}
\end{equation}
\end{lemma}
\begin{proof}
Let
\begin{equation}
\begin{split}
W(l,\bm{x}_0):=\sup_{\pi\in \Pi}&\max\big\{V(\bm{x}_{\bm{x}_0}^{\pi}(l)),\sup_{i\in [0,l)\cap \mathbb{N}}\max_{j\in \{1,\ldots,n_0\}} h_{0,j}(\bm{x}_{\bm{x}_0}^{\pi}(i))\big\}.
\end{split}
\end{equation}
We will prove that for $\epsilon>0$, $|W(l,\bm{x}_0)-V(\bm{x}_0)|<\epsilon$.

According to the definition of $V(\bm{x}_0)$, i.e. \eqref{vf}, for any $\epsilon_1$, there exists an input policy  $\pi'$ such that
\[V(\bm{x}_0)\leq \sup_{i\in \mathbb{N}}\max_{j\in \{1,\ldots,n_0\}}\{h_{0,j}(\bm{x}_{\bm{x}_0}^{\pi'}(i))\}+\epsilon_1.\]

We then introduce two infinite uncertain input policies $\pi_1$ and $\pi_2$ such that
$\pi_1=\{\bm{\pi}_1(i),i\in \mathbb{N}\}$ with $\bm{\pi}_1(j)=\bm{\pi}'(j)$ for $j=0,\ldots,l-1$ and $\pi_2=\{\bm{\pi}_2(i),i\in \mathbb{N}\}$ with $\bm{\pi}(j) =\bm{\pi}'(j+l)$ $\forall j\in \mathbb{N}$.
 Now, let $\bm{y}\in \bm{x}_{\bm{x}_0}^{\pi_1}(l)$, then we obtain that
\oomit{ define $\pi_1^{l_0-1}$ as the restriction of $\pi'$ over $[0,l_0)\cap \mathbb{N}$ and $\pi_2$ as the restriction of $\pi'$ over $[l_0,+\infty)\cap \mathbb{N}$, i.e. $\pi_1^{l_0-1}=\{\bm{d}'(0),\ldots,\bm{d}'(l_0-1)\}$ and $\pi_2=\{\bm{d}_2(l), l=0,\ldots,+\infty)$ with $\bm{d}_2(l)=\bm{d}'(l+l_0)$, and }
\begin{equation*}
\label{w}
\begin{split}
&W(l,\bm{x}_0)\geq\max\big\{V(\bm{y}),\sup_{i\in [0,l)\cap \mathbb{N}}\max_{j\in \{1,\ldots,n_0\}} h_{0,j}(\bm{x}_{\bm{y}}^{\pi_1}(i))\big\}\\
&\geq \max\big\{\sup_{i\in [l,+\infty)\cap \mathbb{N}}\max_{j\in \{1,\ldots,n_0\}}\{h_{0,j}(\bm{x}_{\bm{x}_0}^{\pi_2}(i-l))\},\sup_{i\in [0,l)\cap\mathbb{N}}\max_{j\in \{1,\ldots,n_0\}}\{h_{0,j}(\bm{x}_{\bm{x}_0}^{\pi_1}(i))\}\big\}\\
&=\max\big\{\sup_{i\in [l,+\infty)\cap \mathbb{N}}\max_{j\in \{1,\ldots,n_0\}}\{h_{0,j}(\bm{x}_{\bm{x}_0}^{\pi'}(i))\},\sup_{i\in [0,l)\cap \mathbb{N}}\max_{j\in \{1,\ldots,n_0\}} \{h_{0,j}(\bm{x}_{\bm{x}_0}^{\pi'}(i))\}\big\}\\
&=\sup_{i\in \mathbb{N}}\max_{j\in \{1,\ldots,n_0\}}\{h_{0,j}(\bm{x}_{\bm{x}_0}^{\pi'}(i))\}\\
&\geq V(\bm{x}_0)-\epsilon_1.
\end{split}
\end{equation*}
Therefore,
\begin{equation}
\label{geq}
V(\bm{x}_0)\leq W(l, \bm{x}_0)+\epsilon_1.
\end{equation}

On the other hand, for any $\epsilon_1>0$, there exists a $\pi_1\in \Pi$ 
such that $W(l, \bm{x}_0)\leq \max\big\{V(\bm{x}_{\bm{x}_0}^{\pi_1}(l)),\sup_{i\in [0,l)\cap \mathbb{N}}\max_{j\in \{1,\ldots,n_0\}} \{h_{0,j}(\bm{x}_{\bm{x}_0}^{\pi_1}(i))\}\big\}+\epsilon_1$, by the definition of $W(l, \bm{x}_0)$.
Also, by the definition of $V(\bm{x}_0)$, i.e. \eqref{vf}, for any $\epsilon_1>0$, there exists a $\pi_2$ such that
\[ V(\bm{y})\leq \sup_{i\in \mathbb{N}}\max_{j\in \{1,\ldots,n_0\}} \{h_{0,j}(\bm{x}_{\bm{y}}^{\pi_2}(i))\}+\epsilon_1,\]
where $\bm{y}=\bm{x}_{\bm{x}_0}^{\pi_1}(l)$.
  We define $\pi\in \Pi$ such that $\bm{\pi}(i) = \bm{\pi}_1(i)$ for $i=0,\ldots,l-1$ and $\bm{\pi}(i+l)=\bm{\pi}_2(i)$ for $\forall i\in \mathbb{N}$.
Then, it follows
\begin{equation}
\label{leq}
\begin{split}
W(l, \bm{x}_0)&\leq 2\epsilon_1+\max\{\sup_{i\in \mathbb{N}\cap [l,\infty)}\max_{j\in \{1,\ldots,n_0\}} \{h_{0,j}(\bm{x}_{\bm{y}}^{\pi_2}(i-l))\}, \\
   & \quad \quad \quad \quad \sup_{i\in [0,l)\cap \mathbb{N}}\max_{j\in \{1,\ldots,n_0\}}\{ h_{0,j}(\bm{x}_{\bm{x}_0}^{\pi_1}(i))\}\}\\
&\leq \sup_{i\in [0,+\infty)\cap \mathbb{N}}\max_{j\in \{1,\ldots,n_0\}} \{h_{0,j}(\bm{x}_{\bm{x}_0}^{\pi}(i))\}+2\epsilon_1\\
&\leq V(\bm{x}_0) +2\epsilon_1.
\end{split}
\end{equation}

Combining \eqref{geq} and \eqref{leq}, we finally have $|V(\bm{x}_0)-W(l,\bm{x}_0)|\leq \epsilon=2\epsilon_1$, implying that $V(\bm{x}_0)=W(l,\bm{x}_0)$ since $\epsilon_1$ is arbitrary. This completes the proof.
\end{proof}

Based on Lemma \ref{dp} stating that the value function $V(\bm{x}_0)$ complies with the dynamic programming principle \eqref{dp}, we derive a central equation of this paper, 
which is formulated formally in Theorem \ref{equations}.
\begin{theorem}
\label{equations}
The value function $V(\bm{x}_0):\mathbb{R}^n\mapsto \mathbb{R}$ in \eqref{vf} is a  solution to the equation
\begin{equation}
\label{eq}
\begin{split}
\min\big\{\inf_{\bm{d}\in D}(&V(\bm{x}_0)-V(\bm{f}(\bm{x}_0,\bm{d}))),V(\bm{x}_0)-\max_{j\in \{1,\ldots,n_0\}} h_{0,j}(\bm{x}_0)\big\}=0.\\
\end{split}
\end{equation}
\end{theorem}
\begin{proof}
It is evident that \eqref{eq} is derived from \eqref{dp1} when $l=1$.
\end{proof}

According to Theorem \ref{equations}, we conclude that \textit{if there does not exist a solution to \eqref{eq}, the robust nontermination set $\mathcal{R}_0$ is empty.} Moreover, according to Theorem \ref{equations}, $V(\bm{x}_0)$ as defined in \eqref{vf} is a solution to \eqref{eq}.
Note that the solution to \eqref{eq} may be not unique, and we do not go deeper into this matter in this paper. However, any solution to \eqref{eq} forms an inner-approximation of the maximal robust nontermination set, as stated in Corollary \ref{upper}.
\begin{corollary}
\label{upper}
For any function $u(\bm{x}_0): \mathbb{R}^n \mapsto \mathbb{R}$ satisfying \eqref{eq}, $\{\bm{x}_0\in \mathbb{R}^n\mid u(\bm{x}_0)\leq 0\}$ is an inner-approximation of the maximal robust nontermination set $\mathcal{R}_0$, i.e. $\{\bm{x}_0\in \mathbb{R}^n\mid u(\bm{x}_0)\leq 0\}\subset \mathcal{R}_0$.
\end{corollary}
\begin{proof}
Let $u(\bm{x}_0): \mathbb{R}^n \mapsto \mathbb{R}$ be a solution to \eqref{eq}. It is evident that
 $u(\bm{x}_0)$ satisfies the constraints:
\begin{equation}
\label{upper1}
\left\{
\begin{array}{lll}
u(\bm{x}_0)-u(\bm{f}(\bm{x}_0,\bm{d}))\geq 0, &\forall \bm{x}_0\in \mathbb{R}^n, \forall \bm{d}\in D,\\
u(\bm{x}_0)-h_{0,j}(\bm{x}_0)\geq 0, &\forall \bm{x}_0\in \mathbb{R}^n, \forall j\in \{1,\ldots,n_0\}\\
\end{array}
\right.
\end{equation}

Assume $\bm{x}'_0\in \{\bm{x}_0\mid u(\bm{x}_0)\leq 0\}$.
According to \eqref{upper1}, we have that for $\forall \pi\in \Pi$, $\forall l \in \mathbb{N}$ and $\forall j \in\{1,\ldots,n_0\}$,
\begin{equation}
\begin{cases}
u(\bm{x}_{\bm{x}'_0}^{\pi}(l+1))&\leq u(\bm{x}_{\bm{x}'_0}^{\pi}(l))\leq u(\bm{x}'_0)\\
h_{0,j}(\bm{x}_{\bm{x}'_0}^{\pi}(l))&\leq u(\bm{x}_{\bm{x}'_0}^{\pi}(l))\leq u(\bm{x}'_0)
\end{cases}.
\end{equation}
Therefore, $\sup_{l\in \mathbb{N}}\max_{j\in \{1,\ldots,n_0\}} \{h_{0,j}(\bm{x}_{\bm{x}'_0}^{\pi}(l))\}\leq u(\bm{x}'_0)\leq 0$, implying that $\bm{x}'_0\in \mathcal{R}_0$. Thus, $\{\bm{x}_0 \in\mathbb{R}^n \mid u(\bm{x}_0)\leq 0\}\subset \mathcal{R}_0$.
\end{proof}

 From Corollary \ref{upper}, it is clear that an approximation of $\mathcal{R}_0$ from inside, i.e. a robust nontermination set, is able to be constructed by addressing \eqref{eq}. The solution to \eqref{eq} could be addressed by grid-based numerical methods such as level set methods \cite{fedkiw2002,mitchell2005}, which are a popular method for interface capturing. Such grid-based methods are prohibitive for systems of dimension greater than four without relying upon specific system structure. Besides, we observe that a robust nontermination set could be searched by solving \eqref{upper1} rather than \eqref{eq}. 
 In the subsection that follows we relax \eqref{upper1} as a sum-of-squares decomposition problem in a semidefinite programming formulation when in \textbf{Program} \ref{alg}, $\bm{f}_i$s
 are polynomials over $\bm{x}$ and $\bm{d}$, state and uncertain input constraints, i.e. $X_j$s and $D$s, are restricted to basic semi-algebraic sets.

\subsection{Semi-definite Programming Implementation}
\label{AA}
In practice, it is non-trivial to obtain a solution $V(\bm{x}_0)$ to \eqref{equations}, and thus non-trivial to gain $\mathcal{R}_0$. In this subsection, thanks to \eqref{upper1} and Corollary \ref{upper}, we present a semi-definite programming based method to solve \eqref{eq} approximately and construct a robust invariant $\Omega$ as presented in Definition \ref{RNS}  when Assumption \ref{assum} holds.
\begin{assumption}
\label{assum}
$\bm{f}_i$, $i=1,\ldots,k$, is polynomial over $\bm{x}$ and $\bm{d}$, $X_j$ and $D$, $j=0,\ldots,k$, are restricted to basic semi-algebraic sets in  \textbf{\bf{Program}} $\ref{alg}$.
\end{assumption}

Firstly, \eqref{upper1} has indicator functions on the expression $u(\bm{x}_0)-u(\bm{f}(\bm{x}_0,\bm{d}))$, which is beyond the capability of the solvers we use. We would like to obtain a constraint by removing indicators according to Lemma \ref{split}.
\begin{lemma}[\cite{ChenHWZ15}]
\label{split}
Suppose $\bm{f}'(\bm{x})=1_{F_1}\cdot\bm{f}'_1(\bm{x})+\cdots+1_{F_{k'}}\cdot \bm{f}'_{k'}(\bm{x})$ and $\bm{g}'(\bm{x})=1_{G_1}\cdot \bm{g}'_1(\bm{x})+\cdots+1_{G_{l'}}\cdot \bm{g}'_{l'}(\bm{x})$, where $\bm{x}\in \mathbb{R}^n$, $k',l'\in \mathbb{N}$, and $F_i, G_j\subseteq \mathbb{R}^n$, $i=1,\ldots,k'$, $j=1,\ldots,l'$. Also, $F_1,\ldots, F_{k'}$ and $G_1,\ldots,G_{l'}$ are respectively disjoint. Then, $\bm{f}'\leq \bm{g}'$ if and only if (pointwise)
\begin{equation}
\begin{split}
&\bigwedge_{i=1}^{k'}\bigwedge_{j=1}^{l'}\big[F_i\wedge G_j\Rightarrow \bm{f}'_i\leq \bm{g}'_j\big]\wedge\\
 &\quad\quad\quad\bigwedge_{i=1}^{k'} \big[F_i\wedge \big(\bigwedge_{j=1}^{l'} \neg G_j\big)\Rightarrow \bm{f}'_i\leq 0\big]\wedge\\
& \quad\quad\quad\quad\quad\quad\quad\bigwedge_{j=1}^{l'}\big[\big(\bigwedge_{i=1}^{k'}\neg F_i\big)\wedge G_j\Rightarrow 0\leq \bm{g}'_j\big].
\end{split}
\end{equation}
\end{lemma}

Consequently, according to Lemma \ref{split}, the equivalent constraint without indicator functions of \eqref{upper1} is equivalently formulated below:
\begin{equation}
\label{upper2}
\begin{split}
&\bigwedge_{i=1}^{k}\big[\forall \bm{d}\in D.~\forall \bm{x}_0\in X_i.~ u(\bm{x}_0)-u( \bm{f}_i(\bm{x}_0,\bm{d}))\geq 0\big]\wedge \\
&\bigwedge_{j=1}^{n_0}\big[\forall \bm{x}_0\in \mathbb{R}^n.~u(\bm{x}_0)-h_{0,j}(\bm{x}_0)\geq 0\big].
\end{split}
\end{equation}




Before encoding \eqref{upper2} in sum-of-squares programming formulation, we
denote the set of sum of squares polynomials over variables $\bm{y}$ by $\mathtt{SOS}(\bm{y})$, i.e.
\[\mathtt{SOS}(\bm{y}):=\{p\in\mathbb{R}[\bm{y}]\mid p=\sum_{i=1}^r q_i^2, q_i\in \mathbb{R}[\bm{y}],i=1,\ldots,r\}.\]
Besides, we define the set $\Omega(X_0)$ of states being reachable from the set $X_0$ within one step computation, i.e.,
\begin{equation}
\label{re}
\Omega(X_0):=\{\bm{x}\mid \bm{x}=\bm{f}(\bm{x}_0,\bm{d}),\bm{x}_0\in X_0,\bm{d}\in D\}\cup X_0,
\end{equation}
which can be obtained by semi-definite programming or linear programming methods as in \cite{lasserre2015,magron2017}. Herein, we assume that it was already given. Consequently, when Assumption \ref{assum} holds and $u(\bm{x})$ in \eqref{upper2} is constrained to polynomial type and is restricted in a ball $B=\{\bm{x}\mid h(\bm{x})\geq 0\}$, where $h(\bm{x})= R-\sum_{i=1}^n x_i^2$ and $\Omega(X_0)\subseteq B$, \eqref{upper2} is relaxed as the following sum-of-squares programming problem:
\begin{equation}
\label{sos}
\begin{split}
&\min_{u,s_{i,l_1}^{X_i},s_{i,l_2}^D,s_{i,l},s'_{1,j}} \bm{c}'\cdot \bm{w}\\
&u(\bm{x})-u(\bm{f}_i(\bm{x},\bm{d}))+\sum_{l_1=1}^{n_i}s^{X_i}_{i,l_1}h_{i,l_1}(\bm{x})+\sum_{l_2=1}^{n_{k+1}}s^D_{i,l_2}h_{k+1,l}(\bm{d})-s_{i,1}h(\bm{x}) \in \mathtt{SOS}(\bm{x},\bm{d}),\\
&u(\bm{x})-h_{0,j}(\bm{x})-s'_{1,j}h(\bm{x})\in \mathtt{SOS}(\bm{x}),\\
&i=1,\ldots,k; j=1,\ldots,n_0,
\end{split}
\end{equation}
where $\bm{c}'\cdot \bm{w}=\int_{B}ud\mu(\bm{x})$,  $\bm{w}$ is the vector of the moments of the Lebesgue measure over $B$ indexded in the same basis in which the polynomial $u(\bm{x})\in \mathbb{R}_d[\bm{x}]$ with coefficients $\bm{c}$ is expressed, $s^{X_i}_{i,l_1}, s_{i,l_2}^D, s_{i,1}\in \mathtt{SOS}(\bm{x},\bm{d})$, $i=1,\ldots,k$, $l_1=1,\ldots,n_i$, $l_2=1,\ldots,n_{k+1}$, $s'_{1,j}\in \mathtt{SOS}(\bm{x})$, $j=1,\ldots,n_0$, are sum-of-squares polynomials of appropriate degree. The constraints that polynomials are sum-of-squares can be written explicitly as linear matrix inequalities, and the objective is linear in the coefficients of the polynomial $u(\bm{x})$; therefore problem \eqref{sos} is reformulated as an semi-definite program, which falls within the convex programming framework and can be solved via interior-points method in polynomial time (e.g., \cite{vandenberghe1996}). Note that the objective of \eqref{sos} facilitate the gain of the less conservative robust nontermination set.

The implementation based on the sum-of-squares program \eqref{sos} is sound but incomplete. Its soundness is presented in Theorem \ref{inner}.
\begin{theorem}[Soundness]
\label{inner}
Let $u(\bm{x})\in \mathbb{R}_d[\bm{x}]$ be solution to \eqref{sos}, then $\{\bm{x}\in B\mid u(\bm{x})\leq 0\}$ is an inner-approximation of $\mathcal{R}_0$, i.e., every possible run of \textbf{\bf{Program}} $\ref{alg}$ starting from a state
  in $\{\bm{x}\in B\mid u(\bm{x})\leq 0\}$ does not terminate.
\end{theorem}
\begin{proof}
Since $u(\bm{x})$ satisfies the constraint in \eqref{sos}, we obtain that $u(\bm{x})$ satisfies according to $\mathcal{S}-$ procedure in \cite{boyd1994}:
\begin{align}
&\bigwedge_{i=1}^{k}\big[\forall \bm{d}\in D.~\forall \bm{x}\in X_i \cap B.~u(\bm{x})-u(\bm{f}_i(\bm{x},\bm{d}))\geq 0\big]\wedge\label{1}\\
&\bigwedge_{j=1}^{n_0}\big[\forall \bm{x}\in B.~u(\bm{x})-h_{0,j}(\bm{x})\geq 0\big].
\label{2}
\end{align}

Due to \eqref{1} and the fact that $\cup_{i=1}^k X_i=\mathbb{R}^n$, we obtain that for $\forall \bm{x}_0\in \{\bm{x}\in B\mid u(\bm{x})\leq 0\}$, $\exists i\in\{1,\ldots,k\}.~~\forall \bm{d}\in D.~~u(\bm{x}_0)-u(\bm{f}_i(\bm{x}_0,\bm{d}))\geq 0,$ implying that
\begin{equation}
\label{inequa}
u(\bm{x}_0)-u(\bm{f}(\bm{x}_0,\bm{d}))\geq 0, \forall \bm{d}\in D.
\end{equation}

Assume that there exist an initial state $\bm{y}_0\in \{\bm{x}\in B\mid u(\bm{x})\leq 0\}$ and an input policy $\pi'$ such that $\bm{x}_{\bm{y}_0}^{\pi'}(l)\in X_0$ does not hold for $\forall l\in \mathbb{N}$. Due to  the fact that \eqref{2} holds, we have the conclusion that $\{\bm{x}\in B\mid u(\bm{x})\leq 0\} \subset X_0$ and thus $\bm{y}_0\in X_0$.  Let $l_0\in \mathbb{N}$ be the first time making $\bm{x}_{\bm{y}_0}^{\pi'}(l)$ violate the constraint $X_0$, i.e., $\bm{x}_{\bm{y}_0}^{\pi'}(l_0)\notin X_0$ and $\bm{x}_{\bm{y}_0}^{\pi'}(l)\in X_0$ for $l=0,\ldots,l_0-1$. Also, since $\Omega(X_0)\subset B$, \eqref{inequa} and \eqref{2}, where $\Omega(X_0)$ is defined in \eqref{re}, we derive that $\bm{x}_{\bm{y}_0}^{\pi'}(l_0-1)\in \{\bm{x}\in B\mid u(\bm{x})\leq 0\}$ and $u(\bm{x}_{\bm{y}_0}^{\pi'}(l_0))>0$, which contradicts \eqref{inequa}. Thus, every possible run of \textbf{Problem} \ref{alg} initialized in $\{\bm{x}\in B\mid u(\bm{x})\leq 0\}$ will live in $\{\bm{x}\in B\mid u(\bm{x})\leq 0\}$ forever while respecting $X_0$.


Therefore,  the conclusion in Theorem \ref{inner} is justified.
\end{proof}



\section{Experiments}
\label{ex}
In this section we evaluate the performance of our method built upon  the semi-definite program \eqref{sos}. Examples \ref{ex1} and \ref{ex2} are constructed to illustarte the soundness of our method. Example \ref{ex3} is used to evaluate the scalability of our method in dealing with \textbf{Problem} 1 . The parameters that control the performance of our approach in applying \eqref{sos} to these three examples are presented in Table \ref{table}. All computations were performed on an i7-7500U 2.70GHz CPU with 32GB RAM running Windows 10. For numerical implementation, we formulate the sum of squares problem \eqref{sos} using the MATLAB package YALMIP\footnote{It can be downloaded from \url{https://yalmip.github.io/}.} \cite{lofberg2004} and use Mosek\footnote{For academic use, the software Mosek can be obtained free from \url{https://www.mosek.com/}.} \cite{mosek2015mosek} as a semi-definite programming solver.


\begin{table}[h!]
\begin{center}
\begin{tabular}{|l|r|r|r|r|r|r|}
  \hline
   Ex.&$d_h$&$d_{s_{i,l_1}^{X_i}}$&$d_{s_{i,l_2}^D}$&$d_{s_{i,l}}$&$d_{s'_{1,j}}$&$Time$\\\hline
   1&14&14&14&14 &14&11.30\\\hline
   1&16&16&16&16&16 &28.59\\\hline
   2&6&12&12&12&6 & 9.06\\\hline
   2&8&16&16&16&8 & 65.22\\\hline
   2&10&20&20&20&10 &123.95\\\hline
   2&12&24&24&24&12 &623.95\\\hline
   4&4&4&4&4&4&  58.56     \\\hline
   4&5&4&4&4&4&  60.02     \\\hline
   \end{tabular}
\end{center}
\caption{\textit{Parameters and performance of our implementations on the examples presented in this section.  $d_u, d_{s_{i,l_1}^{X_i}}, d_{s_{i,l_2}^D}, d_{s_{i,l}}, d_{s'_{1,j}}$: the degree of the polynomials $u, s_{i,l_1}^{X_i}, s_{i,l_2}^D, s_{i,l}, s'_{1,j}$ in \eqref{sos}, respectively, $i=1,\ldots,k$, $l_1=1,\ldots,n_i$, $l_2=1,\ldots,n_{k+1}$, $j=1,\ldots,n_0$; $Time$: computation times (seconds).} }
\label{table}
\end{table}

\begin{example}
\label{ex1}
This simple example is mainly constructed to illustrate the difference between \textbf{Program} \ref{alg} taking uncertain inputs into account and free of disturbances. In both cases, \textbf{Program} \ref{alg} is composed of a single loop without switch-case type in loop body, i.e. $k=1$ and $X_1=\mathbb{R}^2$.

In case that $\bm{f}_1(x,y)=(0.4x+0.6y;dx+0.9y)$, $X_0=\{(x,y)\mid  x^2+y^2-1\leq 0\}$ and $D=\{d\mid d^2-0.01\leq 0\}$ in \textbf{Program} \ref{alg}, the inner-approximations of the maximal robust nontermination set $\mathcal{R}_0$ are illustrated in Fig. \ref{fig-one-1}(Left) when $d_u=10$ and $d_u=12$. By visualizing the results in Fig. \ref{fig-one-1}, the inner-approximation obtained when $d_u=12$ does not improve the one when $d_u=10$ a lot. Although there is a gap between the inner-approximations obtained via our method and the set $\mathcal{R}_0$, it is not big. 

In the ideal implementation of \textbf{Program} \ref{alg}, that is, $d$ in the loop body is a fixed nominal value, there will exists some initial conditions such that \textbf{Program} \ref{alg} in the real implementation may violate the constraint set $X_0$, i.e.  \textbf{Program} \ref{alg} may terminate. We use $d=0$ as an instance to illustrate such situation. The difference between termination sets is visualized in Fig. \ref{fig-one-1}(Right). The robust nontermination set in case of $d\in [-0.1,0.1]$ is smaller than the nontermination set when $d=0$. Note that from Fig. \ref{fig-one-3}, we observe that the inner-approximation obtained by our method when $d_u=10$ can approximate $\mathcal{R}_0$ very well.

\begin{figure}[!h]
\begin{minipage}[t]{0.49\linewidth}
\center
\includegraphics[width=2.4in,height=1.6in]{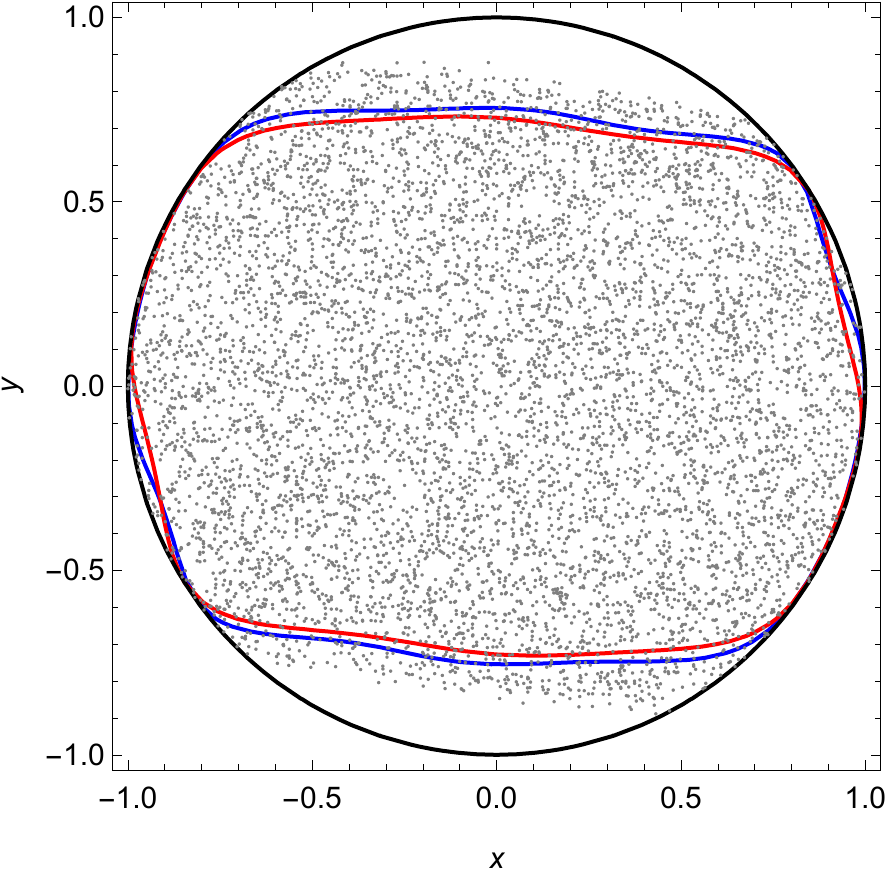}
\end{minipage}
\begin{minipage}[t]{0.49\linewidth}
\center
\includegraphics[width=2.5in,height=1.6in]{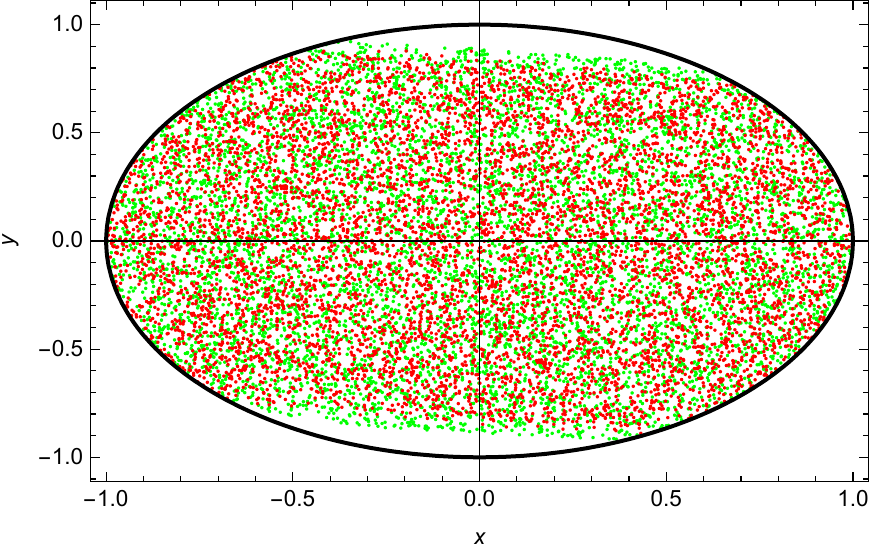}
\end{minipage}
\caption{Computed robust nontermination sets for Example \ref{ex1}.\small{Left: (Blue and Red curves -- the boundaries of the computed robust nontermination set $\mathcal{R}_0$ when $d_u=10$ and $d_u=12$, respectively; Gray points -- the approximated robust nontermination set via numerical simulation techniques; Black curve -- the boundary of $X_0$.) Right: (Green and red points -- the approximated (robust) nontermination sets via numerical simulation techniques for \textbf{Program} \ref{alg} without and with disturbance inputs, respectively; Black curve -- the boundary of $X_0$.)}}
\label{fig-one-1}
\end{figure}

\begin{figure}[!h]
\center
\begin{minipage}[t]{0.49\linewidth}
\includegraphics[width=2.5in,height=1.6in]{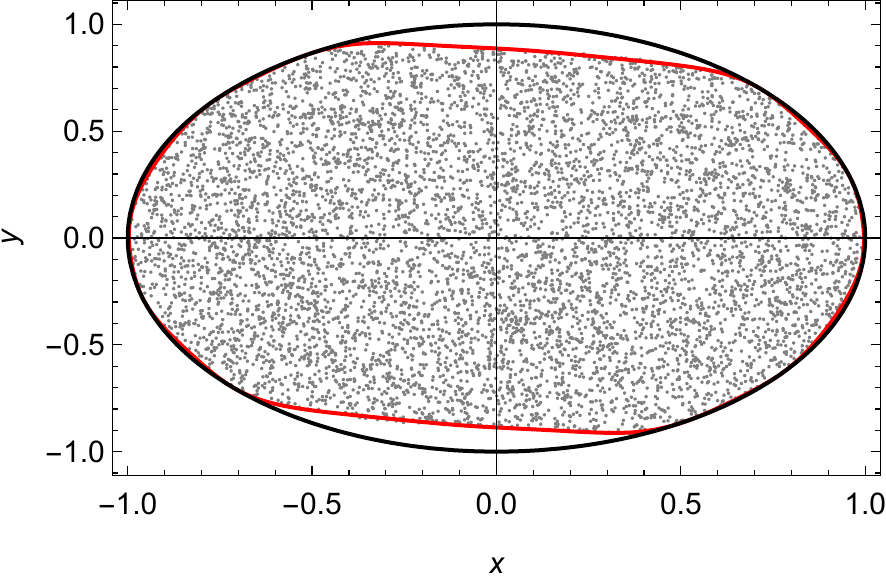}
\caption{Nontermination set estimation for Example \ref{ex1}. \small{(Black and Red curves: the boundaries of $X_0$ and the computed robust nontermination set $\mathcal{R}_0$ when $d_u=16$, respectively; Gray points -- the approximated robust nontermination set via numerical simulation techniques.)}}
\label{fig-one-3}
\end{minipage}
\begin{minipage}[t]{0.49\linewidth}
\includegraphics[width=2.5in,height=1.6in]{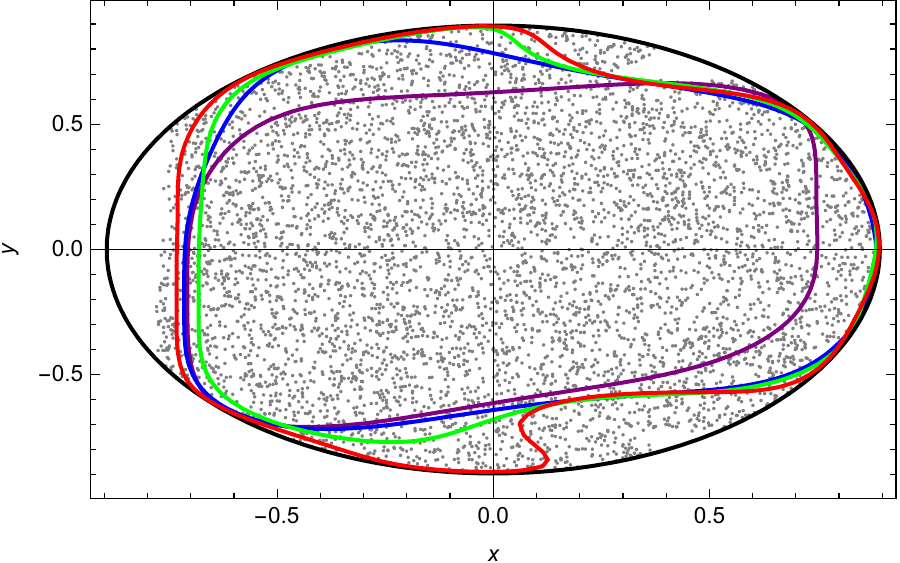}
\caption{Robust nontermination sets for Example \ref{ex2}. \small{Black, Purple, Blue, Green and Red curves: the boundaries of $X_0$ and the computed robust nontermination sets $\mathcal{R}_0$ when $d_u=6,8,10,12$, respectively; Gray points -- the approximated robust nontermination set via numerical simulation techniques.}}
\label{fig-one-4}
\end{minipage}
\end{figure}
\end{example}

\begin{example}
\label{ex2}
In this example we consider \textbf{Program} \ref{alg} with switch-case type in the loop body, where $\bm{f}_1(x,y)=(x;(0.5+d)x-0.1y)$, $\bm{f}_2(x,y)=(y;0.2x-(0.1+d)y+y^2)$, $X_0=\{(x,y)\mid  x^2+y^2-0.8\leq 0\}$, $X_1=\{(x,y)\mid 1-(x-1)^2-y^2\geq 0\}$, $X_2=\{(x,y)\mid -1+(x-1)^2+y^2<0\}$ and $D=\{d\mid d^2-0.01\leq 0\}$. The inner-approximations computed by solving \eqref{sos} when $d_u=8,10$ and $12$ respectively are illustrated in Fig. \ref{fig-one-4}. By comparing these results, we observe that polynomials of higher degree facilitate the gain of less conservative estimation of the set $\mathcal{R}_0$.


\end{example}

\begin{example}
\label{ex3}
In this example,  we consider \textbf{Program} \ref{alg} with seven variables $\bm{x}=(x_1,x_2,x_3,x_4,x_5,x_6,x_7)$ and illustrate the scalability of our approach. In \textbf{Program} \ref{alg}, $\bm{f}_1(\bm{x})=((0.5+d)x_1;0.8x_2;0.6x_3+0.1x_6;x_4;0.8x_5;0.1x_2+x_6;0.2x_2+0.6x_7);$, $\bm{f}_2(\bm{x})=(0.5x_1+0.1x_6;(0.5+d)x_2;x_3;0.1x_1+0.4x_4;0.2x_1+x_5;x_6;0.1x_1+x_7)$, $X_0=\{\bm{x}\mid  \sum_{i=1}^7x_i^2-1\geq 0\}$, $X_1=\{\bm{x}\mid x_1+x_2+x_3-x_4-x_5-x_6-x_7\geq 0\}$, $X_2=\{(x,y)\mid x_1+x_2+x_3-x_4-x_5-x_6-x_7<0\}$ and $D=\{d\mid d^2-0.01\leq 0\}$. From the computation times listed Table \ref{table}, we conclude that although the computation time increases with the number of variables increasing, our method may deal with problems with many variables, especially for the cases that the robust nontermination set formed by a polynomial of low degree fulfills certain needs in real applications. Note that numerical simulation techniques suffers from the curse of dimensionality and thus can not apply to this example since this example has seven variables, we just illustrate the results computed by our method based on \eqref{sos} in Fig. \ref{fig-one-5}.

\begin{figure}[!h]
\center
\includegraphics[width=3.5in,height=2.0in]{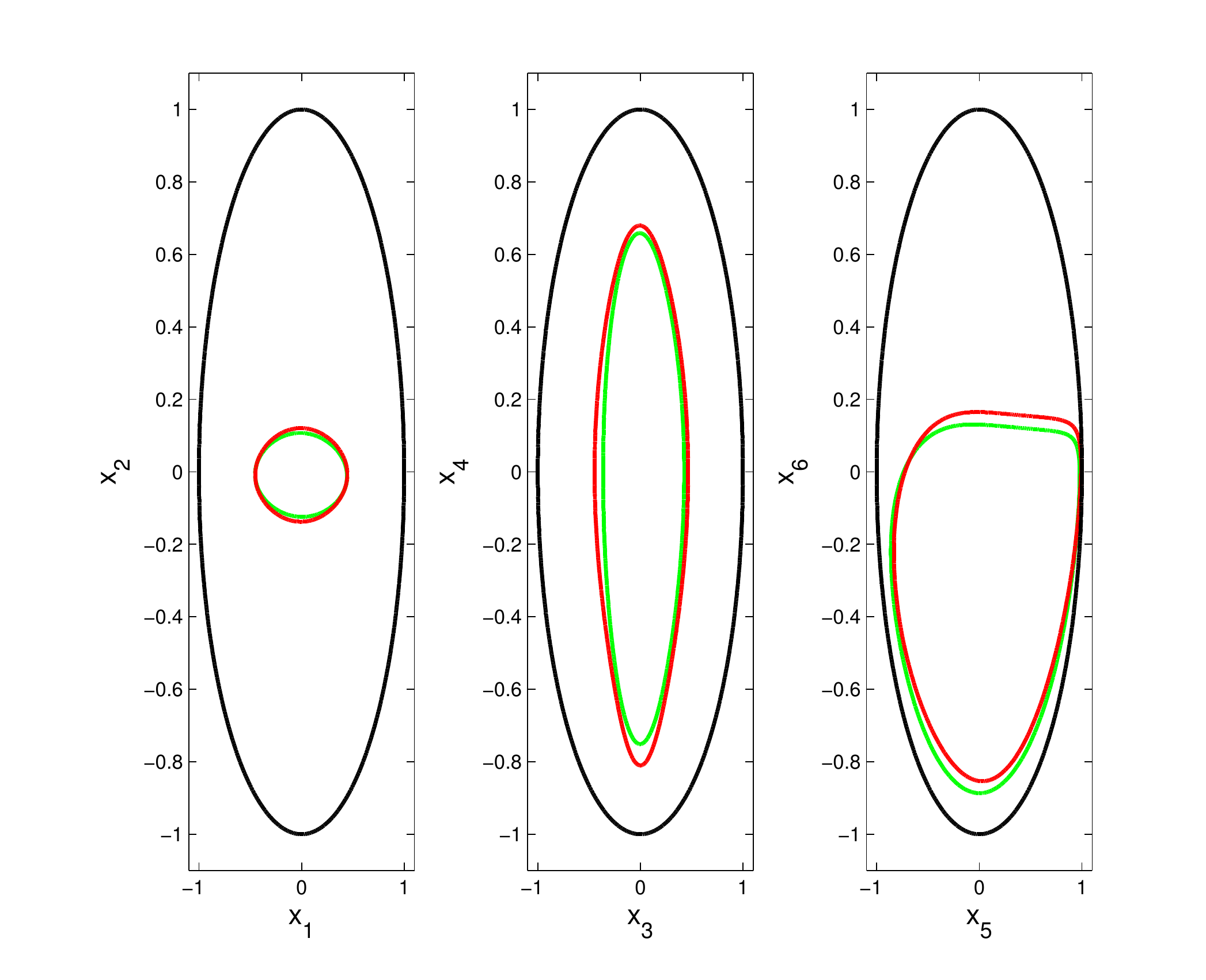}
\caption{Computed robust nontermination sets for Example \ref{ex3}.\small{(Black, Red and Green curves -- the boundaries of $X_0$ and the cross-sections (from left to right: $x_3=x_4=x_5=x_6=x_7=0$, $x_1=x_2=x_5=x_6=x_7=0$ and $x_1=x_2=x_3=x_4=x_7=0$)~of the computed robust nontermination sets $\mathcal{R}_0$ when $d_u=5$ and $d_u=4$, respectively.)}}
\label{fig-one-5}
\end{figure}

\end{example}

\section{Related Work}
\label{RW}
Methods for proving nontermination of programs have recently been studied actively.
\cite{gupta2008} uses a characterization of nontermination by recurrence sets of states that is visited infinitely often along the path. A recurrence set exists iff a program is non-terminating. To find recurrence sets they provide a method based on constraint solving. Their method is only applicable to programs with linear integer arithmetic and does not support non-determinism and is implemented in the tool \textsc{Tnt}. \cite{Chen2014} proposes a method combining closed recurrence sets with counterexample-guided underapproximation for disproving termination. This method, implemented in the tool \textsc{T2},  relies heavily on suitable safety provers for the class of programs of interest, thus rendering an application of their method to nonlinear programs difficult. Further, \cite{cook2014}  introduces live abstractions combing with closed recurrence sets to disprove termination. However, this method, implemented in the tool \textsc{Anant}, is only suitable for disproving non-termination in the form of lasso for programs of finite control-flow graphs.

There are also some approaches exploiting theorem-proving techniques to prove nontermination, e.g., \cite{velroyen2008} presents a method for disproving non-termination of Java programs based on theorem proving and generation of invariants. This method is implemented in \textsc{Invel}, which is restricted to deterministic programs with unbounded integers and single loops. \textsc{Aprove} \cite{giesl2014} uses  SMT solving to prove nontermination of Java programs \cite{brockschmidt2011}. The application of this method requires either singleton recurrence sets or loop conditions being recurrence sets in the programs of interest. \cite{Larraz2014}  disproves termination based on MaxSMT-based invariant generation, which is implemented in the tool \textsc{Cppinv}. This method is limited to linear arithmetic as well.

Besides, \textsc{TRex} \cite{harris2010} integrates existing non-termination proving approaches to develop compositional analysis algorithms for detecting non-termination in multithreaded programs. Different from the method in \textsc{TRex} targeting sequential code, \cite{atig2012} presents a nontermination proving technique for multi-threaded programs via a reduction to nontermination reasoning for sequential programs.  \cite{liu2014} investigates the termination problems of multi-path polynomial programs with equational loop guards and discovering nonterminating inputs for such programs. It shows that the set of all strong non-terminating
inputs and weak non-terminating inputs both correspond to the real varieties of certain polynomial ideals. Recently,\cite{kuwahara2015} proposes a method combining higher-order model checking with predicate abstraction and CEGAR for disproving nontermination of higher-order functional programs. This method reduces the problem of disproving non-termination to the problem of checking a certain branching property of an abstract program, which can be solved by higher-order model checking.

Please refer to \cite{YZZX10,CPR11} for detailed surveys on termination and nontermination analysis of programs.

As opposed to above works without considering robust nontermination, by taking disturbances such as round-off errors in performing numerical implementation of computer programs into account, this paper propose a systematic approach for proving robust nontermination  of a class of computer programs, which are composed of a single loop with a possibly complicated switch-case type loop body and encountered often in current embedded systems. The problem of robust conditional nontermination is reduced to a problem of solving a single equation derived via dynamic programming principle, and semi-definite programs could be employed to solve such optimal control problem efficiently in some situations. 

The underlying idea in this work is in sprit analogous to that in \cite{roozbehani2013}, which is pioneer in proposing a systematic framework to conduct verification of numerical software based on Lyapunov invariance in control theroy. Our method for conducting (robust) verification of numerical software falls within the framework proposed in \cite{roozbehani2013}. The primary contribution of our work is that we systematically investigate a class of computer programs and reduce the nontermination problem for such computer programs to a mathematical equation, thus resulting in an efficient nontermination verification method, as indicated in Introduction \ref{Int}.

\section{Conclusion and Future Work}
\label{con}
In this paper we presented a system-theoretic framework to numerical software analysis and considered the problem of conditional robust non-termination analysis for a class of computer programs composed of a single loop with a possibly complicated switch-case type loop body, which is encountered often in real-time embedded systems. The maximal robust nontermination set of initial configurations in our method was characterized  by a solution to a mathematical equation. Although it is non-trivial to solve gained equation, in the case of polynomial assignments in the loop body and basic semi-algebraic sets in \textbf{Program} \ref{alg}, the equation could be relaxed as a semi-definite program, which falls within the convex programming framework and can be solved efficiently via interior point methods. Finally, we have reported experiments with encouraging results to demonstrate the merits of our method.


However, there are a lot of works remaining to be done. For instance, \oomit{a semi-definite programming relaxation to solve the equation in this paper is sound but incomplete since a solution to the constructed equation is guaranteed to exist in the family of lower semicontinuous functions if the maximal robust nontermination set is not empty, and a continuous solution is not guaranteed. As pointed out previously, a lower semicontinuous function could be approximated by a polynomial function with desired accuracy, we would like to investigate such approximation based on semi-definite programming or linear programming formulation in our future work. Besides,} the semi-definite programming solver is implemented with floating point computations, we have no absolute guarantee on the results it provides. In future work, we need a sound and efficient verification procedure such as that presented in
\cite{platzer2009,sturm2011,LinWYZ14,roux2016} that is able to check the result from the solver and help us decide whether the result is qualitatively correct. Besides, the presented work can be extended in several directions, these include robust nontermination analysis for computer programs with nested loops and robust invariant generations with or without constraints \cite{sankaranarayanan,kapur2006,liu2011}.



\bibliographystyle{abbrv}
\bibliography{reference}

\end{document}